\RequirePackage[leqno]{amsmath}
\documentclass[leqno]{amsart}
\usepackage[utf8]{inputenc}     
\usepackage{multicol}
\usepackage{color, graphics}
\usepackage{pgfplots}
\usepackage{graphicx, multicol, latexsym, amsmath, amssymb}
\usepackage{amsfonts}
\usepackage{enumitem}
\setlist[itemize]{topsep=0pt,after=\vspace{1.5\baselineskip}}
 \usepackage[colorlinks=true]{hyperref}
\usepackage{empheq}
\usepackage[T1]{fontenc}
\usepackage{tabularx}
\usepackage[leftcaption]{sidecap}
\sidecaptionvpos{figure}{m}
\usepackage{tikz}
\usepackage{subfigure}
\usepackage[margin=0.5in]{geometry}
\usetikzlibrary{backgrounds,automata}

\setlist[itemize]{noitemsep, topsep=0pt}

\def\R{\mathbb R}

\def\R{\mathbb R}  
\def\TM{T_{max}} 
\def
\@cite
#1#2{[{{\bfseries #1}\if@tempswa , #2\fi}]}
\newtheorem{theorem}{Theorem}[section]

\newtheorem{lemma}[theorem]{Lemma}

\newtheorem{remark}{Remark}

\title[Boundedness in an attraction-repulsion chemotaxis system with consumption] 
      {
Boundedness in a chemotaxis system with consumed chemoattractant and produced chemorepellent}
\author[S. Frassu and G. Viglialoro]{}
\keywords{Chemotaxis, Global existence, Boundedness, Nonlinear production. \\
	\textit{$^*$Corresponding author}: giuseppe.viglialoro@unica.it}
\begin{document}
\maketitle

\centerline{\scshape Silvia Frassu \and Giuseppe Viglialoro$^{*}$}
\medskip
{
 \medskip
 \centerline{Dipartimento di Matematica e Informatica}
 \centerline{Universit\`{a} di Cagliari}
 \centerline{Via Ospedale 72, 09124. Cagliari (Italy)}
 \medskip
}
\bigskip
\begin{abstract}
We study this zero-flux attraction-repulsion chemotaxis model,  with linear and superlinear production $g$ for the chemorepellent and sublinear rate $f$ for the chemoattractant:
\begin{equation}\label{problem_abstract}
\tag{$\Diamond$}
\begin{cases}
u_t= \Delta u - \chi \nabla \cdot (u \nabla v)+\xi \nabla \cdot (u \nabla w)  & \text{ in } \Omega \times (0,T_{max}),\\
v_t=\Delta v-f(u)v  & \text{ in } \Omega \times (0,T_{max}),\\
0= \Delta w - \delta w + g(u)& \text{ in } \Omega \times (0,T_{max}).
\end{cases}
\end{equation}
In this problem, $\Omega$ is a bounded and smooth domain of $\R^n$, for $n\geq 1$, $\chi,\xi,\delta>0$, $f(u)$ and $g(u)$ reasonably regular functions generalizing the prototypes $f(u)=K u^\alpha$ and $g(u)=\gamma u^l$, with $K,\gamma>0$ and proper $ \alpha, l>0$. Once it is indicated that any sufficiently smooth $u(x,0)=u_0(x)\geq 0$ and $v(x,0)=v_0(x)\geq 0$ produce a unique classical and nonnegative solution $(u,v,w)$ to \eqref{problem_abstract}, which is defined in $\Omega \times (0,T_{max})$,  we establish that for any such $(u_0,v_0)$, the life span $\TM=\infty$ and $u, v$ and $w$ are uniformly bounded in $\Omega\times (0,\infty)$, (i) for $l=1$, $n\in \{1,2\}$, $\alpha\in (0,\frac{1}{2}+\frac{1}{n})\cap (0,1)$ and any $\xi>0$, (ii) for $l=1$, $n\geq 3$, $\alpha\in (0,\frac{1}{2}+\frac{1}{n})$ and $\xi$ larger than a quantity depending on 
$\chi \lVert v_0 \rVert_{L^\infty(\Omega)}$, (iii) for $l>1$ any $\xi>0$, and in any dimensional settings. Finally, an indicative analysis about the effect by logistic and repulsive actions on chemotactic phenomena is proposed by comparing the results herein derived for the linear production case with those in \cite{LankeitWangConsumptLogistic}. 
\end{abstract}
\section{Presentation of the model}\label{IntroSection}
This article is dedicated to the following Cauchy boundary problem
\begin{equation}\label{problem}
\begin{cases}
u_t= \Delta u - \chi \nabla \cdot (u \nabla v)+\xi \nabla \cdot (u \nabla w)  & \text{ in } \Omega \times (0,T_{max}),\\
v_t=\Delta v-f(u)v  & \text{ in } \Omega \times (0,T_{max}),\\
0= \Delta w - \delta w + g(u)& \text{ in } \Omega \times (0,T_{max}),\\
u_{\nu}=v_{\nu}=w_{\nu}=0 & \text{ on } \partial \Omega \times (0,T_{max}),\\
u(x,0)=u_0(x), \; v(x,0)=v_0(x) & x \in \bar\Omega,
\end{cases}
\end{equation}
defined  in a bounded and smooth domain $\Omega$ of $\R^n$, with $n\geq 1$, $\chi, \xi, \delta>0$ and some  functions $f=f(s)$ and $g=g(s)$, sufficiently regular in their argument $s\geq 0$, and further regular initial data  $u_0(x)\geq 0$ and $v_0(x)\geq 0.$  Additionally,  the subscript $\nu$ in $(\cdot)_\nu$ indicates the outward normal derivative on $\partial \Omega$, whereas $T_{max}$ the maximum time up to which solutions to the system are defined. 

The consideration of model \eqref{problem} comes, essentially, from a natural coupling of two widely studied chemotaxis systems: the classical Keller--Segel model (\cite{K-S-1970,Keller-1971-MC,Keller-1971-TBC}) idealizing aggregation phenomena in situations where certain cells (populations, organisms) are attracted by a signal they themselves absorb, and a repulsive counterpart, where the same cells are repelled in response to another substance emitted by them. More precisely, if $u=u(x,t)$ is used to denote the population density of these cells at the 
position $x$ and at the time $t$, and $v=v(x,t)$ and $w=w(x,t)$ stand, respectively,  for the concentration of the attractive and repulsive chemical signals (chemoattractant and chemorepellent), problem \eqref{problem} indicates that: 
(a) the motion of the cells, inside an insulated domain (zero-flux on the border) and initially distributed according to the law of $u_0$, results from the competition between the aggregation/repulsion impact from the cross terms $\chi u \nabla v/\xi u \nabla w$ (increasing for larger sizes of $\chi$ and $\xi$) and the diffusion of the cells (the Laplacian $\Delta u$); (b) the initial signal $v_0$ is spread, $w$ diffuses as well but $v$ (second equation in \eqref{problem}) is consumed with a rate $f(u)$ whereas $w$ (third equation) is proliferated with rate $g(u)$; (c) consumption and production are higher the more the cell density increases. 

Purely intuitive considerations (but below we will give precise references) suggest that this interplay between the factors taking part in model \eqref{problem} might lead to very different situations for the aforementioned cellular movement: from global stabilization and convergence to equilibrium of the cell distribution $u$, to the so-called \textit{chemotactic collapse}, the mechanism resulting in aggregation processes for $u$, eventually blowing up/exploding at finite time. Mathematically, in the first case, solutions $(u,v,w)$ are defined and bounded for all $(x,t)$ in $\Omega \times (0,\infty)$, in the other a finite time $\TM$ exists and $(u,v,w)$ ceases to exist for larger value of 
$\TM$; in particular the component of the solution associated to the particle density becomes unbounded approaching $\TM$, with emergencies of $\delta$-formations. In this research we will derive criteria on the data involved in the initial-boundary value problem \eqref{problem} ensuring that the life span $\TM$ of its solutions is infinity and that, moreover, they are as well bounded. 
\section{Some known results. Claim of the main theorems}
\subsection{A view on the state of the art} 
In the framework of classical Keller--Segel models, as mentioned above, \eqref{problem} is a combination of the signal-production 
\begin{equation}\label{problemOriginalKS}
u_t= \Delta u - \chi  \nabla \cdot (u \nabla v) \quad \textrm{and} \quad 
v_t=\Delta v-v+u,  \quad  \text{ in } \Omega \times (0,T_{max}),
\end{equation}
and signal-absorption 
\begin{equation}\label{problemOriginalKSCosnumption}
u_t= \Delta u - \chi \nabla \cdot (u \nabla v)  \quad \textrm{and} \quad 
v_t=\Delta v-u v,  \quad  \text{ in } \Omega \times (0,T_{max}),
\end{equation}
chemotaxis systems. Even though the equation for $u$ is the same, it is conceivable that the resulting evolution of each initial boundary value problem related to \eqref{problemOriginalKS} and \eqref{problemOriginalKSCosnumption}, must differ from the other, even for same fixed $\chi>0$ and initial data $u_0$ and $v_0$. This is essentially justified by the observation that $v$ increases with $u$ in problem \eqref{problemOriginalKS}, whereas it decreases in \eqref{problemOriginalKSCosnumption}. Let us present some more details concerning this discussion; in particular,  since we will focus on questions tied to classical solutions, in order to better establish our aims, we select only these references, among some others. 
\begin{enumerate}[label=\roman*)]
\item  \label{item1Intro}For problem \eqref{problemOriginalKS}, the production of $v$ may break the natural homogenization process of the cells, especially in terms of the size of $\chi$ related to the aggregation impact, the initial mass of the particle distribution, i.e.,  $m=\int_\Omega u_0(x)dx,$ and the space dimension. Indeed, if in the one-dimensional setting blow-up phenomena are excluded  (see \cite{OsYagUnidim}), in higher dimensions if $m \chi$  surpasses a certain critical value $m_\chi$, the system might present the aforementioned chemotactic collapse, whereas for $m\chi<m_\chi$ no instability appears in the motion of the cells. There are many contributions dedicated to understanding this scenario. In this regard, in \cite{HerreroVelazquez,JaLu,Nagai,WinklAggre} (and references therein cited), the interested reader can find pointers to the rich literature dealing with the existence and properties of global, uniformly bounded or blow-up (local) solutions to the Cauchy problem associated to \eqref{problemOriginalKS}. On the other hand, as far as nonlinear segregation chemotaxis models like those we are considering, when in problem \eqref{problemOriginalKS}  the production $g(u)=u$ is replaced by $g(u)\cong u^l$, with  $0<l<\frac{2}{n}$ ($n\geq1$), uniform boundedness of all its solutions is proved in \cite{LiuTaoFullyParNonlinearProd}. Moreover, by resorting to a simplified parabolic-elliptic version in spatially radial contexts, when the second equation is reduced to $0=\Delta v-\mu(t)+g(u)$, with $g(u)\cong u^l$ and $\mu(t)=\frac{1}{|\Omega|}\int_\Omega g(u(\cdot,t))$, it is known  (see \cite{WinklerNoNLinearanalysisSublinearProduction}) that  the same conclusion on the boundedness continues to be valid for any $n\geq 1$ and $0<l<\frac{2}{n}$, whereas for $l>\frac{2}{n}$ blow-up phenomena may occur.
\item  \label{item2Intro} Conversely to what was discussed for model \eqref{problemOriginalKS}, when the chemical $v$ responsible for gathering processes of the cells is consumed throughout the time, so far no result detecting unbounded solutions to the corresponding initial boundary-value problem to \eqref{problemOriginalKSCosnumption} is available. Such a question seems quite hard to solve, and this does not appear surprising;  indeed, from comparison arguments, the second equation for the chemical immediately ensures uniform boundedness of $v$. Despite that, such a bound by itself is not enough to ensure that classical solutions $(u,v)$ to \eqref{problemOriginalKSCosnumption} emanating from any sufficiently regular initial data $(u_0,v_0) $ are uniformly bounded. Precisely, this holds true in two-dimensional settings (as a combination of the results in \cite{WinklerN-Sto_CPDE}  and \cite{WinklerN-Sto_2d}, where a more general coupled chemotaxis-fluid model is studied) and for $n\geq 3$, provided this smallness assumption is satisfied (\cite{TaoBoun}): $\chi \lVert v_0\lVert_{L^\infty(\Omega)}\leq \frac{1}{6(n+1)}.$ Nevertheless, this condition does not exclude the possibility that solutions emanating from initial data, not satisfying it, may collapse in finite time. Despite that, a way to prevent blow-up of solutions to problem \eqref{problemOriginalKSCosnumption} even for values of $\chi \lVert v_0\lVert_{L^\infty(\Omega)}$ larger than $\frac{1}{6(n+1)}$, is considering logistic sources with strong dampening effect in the equation of the cells, precisely reading
\begin{equation}\label{EquationWithLogistic}
u_t= \Delta u - \chi \nabla \cdot (u \nabla v) +ku-\mu u^2, \quad  \text{ in } \Omega \times (0,T_{max}),\quad k,\mu>0.
\end{equation}
In \cite{LankeitWangConsumptLogistic} it is indeed shown that the resulting Cauchy problem admits classical bounded solutions for arbitrarily large  $\chi \lVert v_0\lVert_{L^\infty(\Omega)}$ provided $\mu$ is also larger than a certain expression depending in an increasing way  on the same $\chi \lVert v_0\lVert_{L^\infty(\Omega)}$.
\end{enumerate}
As far as we know,  a general $n$-dimensional analysis tied to the attraction-repulsion chemotaxis system in the form of  \eqref{problem}, has not been developed yet. Conversely, for $f(u)=g(u)=u$, a fully parabolic attraction-repulsion Stokes system is addressed for the two-dimensional case in \cite{liu2020stabilization}: here, inter alia, boundedness of classical solutions is achieved for any initial data. 
In addition, for model \eqref{problem} where the chemoattractant and chemorepellent are both produced it has been proposed as well in the fully parabolic version in \cite{Luca2003Alzheimer}, for one-dimensional settings and linear proliferation, to describe the aggregation of microglia observed in Alzheimer's disease. In particular, for any $n>1$, the attraction-repulsion system \eqref{problem} with second and third equations replaced by
\begin{equation*}
0= \Delta v -\beta  v +f(u)  \quad \textrm{and} \quad 
0=\Delta w-\delta w+g(u) ,  \quad  \text{ in } \Omega \times (0,T_{max}),
\quad \beta, \delta>0,
\end{equation*}
the following is known in the literature. For linear growths of the chemoattractant and the chemorepellent, $f(u)=\alpha u$, $\alpha>0$, and $g(u)=\gamma u$, $\gamma>0$, 
we have that the value $\xi\gamma-\chi\alpha$, measuring in some sense the difference between the repulsion and attraction contributions, is critical: if $\xi\gamma-\chi\alpha>0$ (repulsion prevails over attraction) all solutions to the model are globally bounded, whereas for $\xi\gamma-\chi\alpha<0$ (attraction prevails over repulsion) unbounded solutions can be constructed: see \cite{GuoJiangZhengAttr-Rep,LI-LiAttrRepuls,TaoWanM3ASAttrRep,VIGLIALORO-JMAA-BlowUp-Attr-Rep,YUGUOZHENG-Attr-Repul} for some details on the issue. On the other hand, for more general production laws, respectively $f$ and $g$ generalizing the prototypes $f(u)=\alpha u^s$, $s>0,$ and $g(u)=\gamma u^r$, $r\geq 1$, we are only aware of the following recent result (\cite{ViglialoroMatNacAttr-Repul}): for every  $\alpha,\beta,\gamma,\delta,\chi>0$,  and $r>s\geq 1$ (resp. $s>r\geq 1$), there exists $\xi^*>0$ (resp. $\xi_*>0$) such that if $\xi>\xi^*$ (resp. $\xi\geq \xi_*$), any sufficiently regular initial datum $u_0(x)\geq 0$ (resp. $u_0(x)\geq 0$ enjoying some smallness assumptions) produces a unique classical and bounded solution. In addition the same conclusion holds true for every  $\alpha,\beta,\gamma,\delta,\chi,\xi>0$, $0<s<1$, $r=1$ and  any sufficiently regular $u_0(x)\geq 0$.
\subsection{Motivations and presentation of the Theorems}
In accordance to what has been discussed above, especially in items  \ref{item1Intro}  and \ref{item2Intro}, we wish to contribute to the analysis of attraction-repulsion Keller--Segel systems by giving answers to questions concerning system \eqref{problem}, to our knowledge, not yet studied. In this sense, we aim at essentially establishing the roles of the chemoattractant and chemorepellent on the motion of the particle density, whose kinetics are not influenced by any smoothing logistic term. Specifically, we will give sufficient conditions on the data of model \eqref{problem} such that the joint actions of the consumed chemoattractant and the produced chemorepellent suffice to provide global and bounded solutions in terms of, or independently of, smallness constraints on $\chi \lVert v_0\lVert_{L^\infty(\Omega)}$. To this scope, these assumptions are fixed
\begin{equation}\label{f}
f,g \in C^1(\R) \quad \textrm{with} \quad   0\leq f(s)\leq Ks^{\alpha}  \textrm{ and } \gamma s^l\leq g(s)\leq \gamma s(s+1)^{l-1},\quad  \textrm{for some}\quad K,\gamma,\alpha>0, l\geq 1 \quad \textrm{and all } s \geq 0,
\end{equation}
and the following results are shown.
\begin{theorem}\label{MainTheorem}
Let $\Omega$ be a smooth and bounded domain of $\mathbb{R}^n$, with $n\geq 1$, and $\chi, \delta$ positive. Moreover, for some $K,\gamma>0$, let $f$ and $g$ fulfill \eqref{f}, respectively with $\alpha \in \left(0, \frac{1}{2}+\frac{1}{n}\right)\cap (0,1)$ and $l=1$. Then there exists $C(n)>0$ such that for any initial data $(u_0,v_0)\in C^0(\bar{\Omega})\times C^{1}(\bar\Omega)$, with $u_0, v_0\geq 0$ on $\bar{\Omega}$, any $\xi>0$ and $n\in\{1,2\}$, or $\xi > {C}(n) \|\chi v_0\|_{L^{\infty}(\Omega)}^\frac{4}{n}$ and $n\geq 3$, 
problem \eqref{problem} admits a unique global and uniformly bounded classical solution.  
\end{theorem}
\begin{theorem}\label{Main1Theorem}
Let $\Omega$ be a  smooth and bounded domain of $\mathbb{R}^n$, with $n\geq 1$, and $\chi,\delta$ positive. Moreover, for some $K,\gamma>0$, let $f$ and $g$ fulfill \eqref{f}, respectively with $\alpha \in \left(0, \frac{1}{2}+\frac{1}{n}\right)\cap (0,1)$ and $l>1$. Then for any $\xi>0$ and any initial data $(u_0,v_0)\in C^0(\bar{\Omega})\times C^{1}(\bar\Omega)$, with $u_0, v_0\geq 0$ on $\bar{\Omega}$,
problem \eqref{problem} admits a unique global and uniformly bounded classical solution.  
\end{theorem}
\begin{remark}\label{RemarkUnifBounCalsSol}
As usual in the nomenclature, in chemotaxis models a global and uniformly bounded classical solution to problem \eqref{problem} is a triplet  of nonnegative functions 
$(u,v,w)\in (C^0(\bar{\Omega}\times [0,\infty))\cap  C^{2,1}(\bar{\Omega}\times (0,\infty)))^3,$ such that for some $q>n$ and $C>0$ this relation holds:
	\begin{equation*}
	\lVert u(\cdot,t) \rVert_{L^\infty(\Omega)}+\lVert v(\cdot,t) \rVert_{W^{1,q}(\Omega)} \leq C \quad \textrm{for all} \quad t\in(0,\infty).
	\end{equation*}
\end{remark}
The remaining part of the paper is structured as follows: In $\S$\ref{PreliminariesSection} some general and well-known preliminaries are given, whereas $\S$\ref{SectionLocalInTime} is focused on the existence of local classical solutions $(u,v,w)$ to problem \eqref{problem} defined in $\Omega\times (0,\TM)$. In particular, crucial properties of these solutions, and how to achieve their uniform-in-time boundedness from their $L^p$-boundedness, for some suitable $p>1$, is analyzed. Successively, in $\S$\ref{EstimatesAndProofSection}, we associate to the local solutions, the functional  $y (t):=\int_\Omega u^p+(\frac{\chi^2}{\gamma})^{p}\int_\Omega |\nabla v|^{2p}$, by means of which the desired uniform-in-time bound is proved; this will allow us to proof our results, also in the same $\S$\ref{EstimatesAndProofSection}. Finally, in $\S$\ref{SectionComparison} we compare \cite[Theorem 1.1]{LankeitWangConsumptLogistic} and Theorem \ref{MainTheorem} as to discuss the boundedness issue for chemotaxis-consumption models with different smoothing reactions: a logistic source and a produced chemorepellent. 
\section{Some preparatory tools}\label{PreliminariesSection}
In this section we summarize some  inequalities and further necessary results. 
\begin{lemma}\label{BoundsInequalityLemmaTecnicoYoung}  
	Let $A,B \geq 0$, $d_1,d_2>0$ and $p>1$. Then for some $d, d_3>0$
	we have 
	\begin{equation}\label{InequalityForFinallConclusion}
	A^{d_1}+B^{d_2}\geq 2^{-d}(A+B)^{d}-d_3, 
	\end{equation}
	and
	\begin{equation}\label{InequalityA+BToPowerP}
	(A+B)^p \leq 2^{p-1}(A^p+B^p).
	\end{equation}
	\begin{proof}
		The proofs can be found, respectively, in \cite[Lemma 
		3.3]{MarrasViglialoroMathNach} and  
		\cite[Theorem 1]{Jameson_2014Inequality}.
	\end{proof}
\end{lemma}
\begin{lemma}\label{InequalityTipoPoincarLemma} 
Let $\Omega$ be a bounded and smooth domain of $\R^n$, $n\geq 1$. For all $\psi\in C^{2}(\bar{\Omega})$, we have
\begin{equation}\label{InequalityLaplacian}
(\Delta \psi)^2 \leq n |D^2 \psi|^2,
\end{equation}
\begin{equation}\label{InequalityGradienHessian}
\lvert D^2\psi \nabla \psi \rvert^2\leq |D^2 \psi|^2\lvert \nabla \psi \rvert^2.
\end{equation}
If, further, $\psi$ satisfies $\psi_{\nu}=0$ on $\partial \Omega$, then for all $p> 1$ and $\eta>0$ one has
\begin{equation}\label{InequalityHessian}
\lVert  \nabla \psi\rVert^{2p+2}_{L^{2p+2}(\Omega)} \leq 2 (4p^2+n)\lVert \psi \rVert^2_{L^\infty(\bar{\Omega})} \lVert \vert \nabla \psi\vert^{p-1}  D^2 \psi\rVert^2_{L^2(\Omega)},
\end{equation}
where $D^2\psi$ represents the Hessian matrix of $\psi$ and $\vert D^2\psi \rvert^2=\sum\limits_{i,j=1}^{n}\psi_{x_ix_j}^2$, whereas for some positive constant $C_\eta$
\begin{equation}\label{NoConvexity}
\int_{\partial\Omega} |\nabla \psi|^{2p-2} (|\nabla \psi|^2)_{\nu} \leq \eta \int_\Omega |\nabla \psi|^{2p-4} |\nabla|\nabla \psi|^2|^2 + C_{\eta} 
\left(\int_\Omega |\nabla \psi|^2\right)^p.
\end{equation}
\begin{proof}
Regard the proof of inequalities \eqref{InequalityLaplacian} and \eqref{InequalityGradienHessian}, we refer the reader to \cite[Lemma 3.1]{MarrasViglialoroMathNach}.  As to \eqref{InequalityHessian}, this is a special case of \cite[Lemma 2.2]{LankeitWangConsumptLogistic}, and relation \eqref{NoConvexity} is derived in \cite[Lemma 2.1 c)]{LankeitWangConsumptLogistic}.
\end{proof}
\end{lemma}
\begin{lemma}\label{EllipticEhrlingSystemLemma} 
Let $\Omega\subset \R^n$, $n\geq 1$, be a bounded and smooth domain and $\delta>0$. Then for any nonnegative $g\in C^1(\bar{\Omega})$, the solution $0\leq \psi\in C^{2,\kappa}(\bar{\Omega})$, $0<\kappa<1$, of the problem
\begin{equation*}
\begin{cases}
0=\Delta \psi+g-\delta \psi & \textrm{in } \Omega,\\
\psi_{\nu}=0 & \textrm{on } \partial \Omega,
\end{cases}
\end{equation*}
has the following property: For any $\hat{c},\sigma>0$ and  $\overline{p}\in(1,\infty)$, there exists $\tilde{c}=\tilde{c}(\sigma,\overline{p}) >0$ such that 
\begin{equation}\label{EhrlingTypeInequalityWithMass} 
\hat{c}\int_\Omega \psi^{\overline{p}+1} \leq \sigma  \int_\Omega g^{\overline{p}+1} +\frac{\tilde{c}}{|\Omega|^{\overline{p}}}\Big( \int_\Omega g\Big)^{\overline{p}+1}.\end{equation}
\end{lemma}
\begin{proof}
A detailed proof of \eqref{EhrlingTypeInequalityWithMass} can be found in \cite[Lemma 3.1]{ViglialoroMatNacAttr-Repul}. (See also \cite[Lemma 2.2]{WinklerHowFar}.)
\end{proof}
\section{Existence of local-in-time classical solutions. From uniform boundedness in $L^p(\Omega)$ to $L^\infty(\Omega).$}\label{SectionLocalInTime}
Let us dedicate ourselves to the existence question of classical solutions to system \eqref{problem}. It is shown that such solutions are at least local and,  additionally, satisfy some crucial estimates.
\begin{lemma}[\rm{Local existence}]\label{LocalExistenceLemma}   
Let $\Omega$ be a bounded and smooth domain of $\R^n$, with $n \geq 1$, $q>n$, $\chi,\delta>0$ and nontrivial
$(u_0,v_0)\in C^0(\bar{\Omega})\times C^{1}(\bar\Omega)$, with $u_0\geq 0$ and $v_0\geq 0$ on $\bar{\Omega}$.  Assume, moreover, that for some $\gamma, K>0$, $f$ and $g$ fulfill \eqref{f}, respectively with  $\alpha\in (0,\frac{1}{2}+\frac{1}{n})\cap (0,1)$ and $l\geq 1$. Then, for any $\xi>0$ there exist $\TM \in (0,\infty]$ and a unique triplet of nonnegative functions $(u,v,w)\in 
 (C^0(\bar{\Omega}\times [0,T_{max}))\cap  C^{2,1}(\bar{\Omega}\times (0,T_{max})))^3,$ such that this dichotomy criterion holds true:
\begin{equation}\label{ExtensibilityCrit}
 \textrm{either}\,\; T_{max}=\infty\; \textrm{or}\; \limsup_{t \rightarrow T_{max}}(\lVert u(\cdot,t)\rVert_{L^\infty(\Omega)}+\lVert v(\cdot,t)\rVert_{W^{1,q}(\Omega)}) =\infty. 
\end{equation} 
In addition, the $u$-component obeys the mass conservation property, i.e. 
\begin{equation}\label{massConservation}
	\int_\Omega u(x, t)dx =\int_\Omega u_0(x)dx=m>0\quad \textrm{for all }\, t \in (0,\TM),
\end{equation}
whilst for some $c_0>0$ the $v$-component is such that 
\begin{equation}\label{Cg}
0 \leq v\leq \lVert v_0\rVert_{L^\infty(\Omega)}\quad \textrm{in}\quad \Omega \times (0,T_{max})\quad \textrm{and}\quad 
\int_\Omega |\nabla v(\cdot, t)|^2\leq c_0 \quad \textrm{on } \,  (0,\TM).
\end{equation}
\begin{proof}
The local solvability  as well as the dichotomy criterion \eqref{ExtensibilityCrit} can be proved by adapting well-established approaches widely used in
the frame of classical chemotaxis models (see for instance \cite[Lemma 1.2]{CiesWink}, \cite[Theorem 3.1]{HorstWink} and \cite[Lemma 3.1]{TaoWanM3ASAttrRep}). Moreover, comparison arguments apply to yield $u, v, w\geq 0$ in $\Omega\times (0,\TM)$ and the first relation in \eqref{Cg}, whereas the mass conservation property follows by integrating over $\Omega$ the first equation of \eqref{problem}, in conjunction with the boundary and initial conditions.

Let us, finally, derive the last claim as follows. We separate the cases $0<\alpha\leq \frac{1}{2}$ and $ \frac{1}{2}<\alpha <\min\{\frac{1}{2}+\frac{1}{n},1\}$. For $0<\alpha\leq \frac{1}{2}$, from the second equation of \eqref{problem}, we have that an integration over $\Omega$, the Young inequality, the bound for $v$ given in \eqref{Cg} and the properties of $f$ in \eqref{f} lead to
\begin{equation}\label{AA}
\begin{split}
\frac{d}{dt}\int_\Omega \lvert \nabla v\rvert^2&=2\int_\Omega \nabla v \cdot \nabla (\Delta v -f(u)v)=-2\int_\Omega (\Delta v)^2+2\int_\Omega f(u)v \Delta v \\&
=-2 \int_\Omega (\Delta v)^2+2 \int_\Omega v(f(u)-1)\Delta v-2\int_\Omega  \lvert \nabla v\rvert^2  
 \leq - \int_\Omega (\Delta v)^2-2\int_\Omega  \lvert \nabla v\rvert^2 + \int_\Omega v^2(f(u)-1)^2\\ & 
\leq -2\int_\Omega  \lvert \nabla v\rvert^2 + \lVert v_0\rVert^2_{L^\infty(\Omega)}K^2\int_\Omega u^{2\alpha}+ 2\lVert v_0\rVert^2_{L^\infty(\Omega)}K \int_\Omega u^{\alpha}+ \lVert v_0\rVert^2_{L^\infty(\Omega)}|\Omega| \quad \textrm{on } (0,\TM).
\end{split}
\end{equation}
Now, since $L^1(\Omega) \subseteq L^{2\alpha}(\Omega)\subseteq L^{\alpha}(\Omega)$, thanks to the mass conservation property \eqref{massConservation} we can find $c_1>0$ such that 
$$ \lVert v_0\rVert^2_{L^\infty(\Omega)}K^2\int_\Omega u^{2\alpha}+ 2\lVert v_0\rVert^2_{L^\infty(\Omega)}K \int_\Omega u^{\alpha}+ \lVert v_0\rVert^2_{L^\infty(\Omega)}|\Omega|\leq c_1\quad \textrm{with } t\in (0,\TM),$$
so that \eqref{AA} reads
\begin{equation*}
\frac{d}{dt}\int_\Omega \lvert \nabla v\rvert^2\leq -2\int_\Omega \lvert \nabla v\rvert^2+c_1 \quad \textrm{on } (0,\TM),
\end{equation*}
and a comparison argument entails $\int_\Omega \lvert \nabla v\rvert^2\leq \max\{\frac{c_1}{2},\int_\Omega \lvert \nabla v_0\rvert^2\}$ for all $t \in (0,\TM).$ 

When, indeed, $\frac{1}{2}<\alpha<\min\{\frac{1}{2}+\frac{1}{n},1\}$, we can pick $\frac{1}{2}<\rho <1-\frac{n}{2}\big(\alpha-\frac{1}{2}\big)$ and set $\zeta=1-\rho-\frac{n}{2}\big(\alpha-\frac{1}{2}\big)>0$. Moreover, through the H\"{o}lder inequality, taking in mind
\eqref{f} and again  \eqref{massConservation}, we have
\begin{equation}\label{Estim_For_f1}
\lVert  f (u(\cdot, t)) \lVert_{L^{\frac{1}{\alpha}}(\Omega)}^{\frac{1}{\alpha}}= \int_\Omega f(u)^{\frac{1}{\alpha}}\leq  K^{\frac{1}{\alpha}}\int_\Omega u\leq K^{\frac{1}{\alpha}}m\quad \textrm{for all } t< \TM.
\end{equation}
As a consequence,  from the representation formula for $v$, we have
\[v(\cdot,t) =e^{t\Delta}v_0 +\int_0^t e^{(t-s)\Delta}f(u(\cdot,s))v(\cdot,s)ds \quad \textrm{for all } \,t\in (0,T_{max}),
\]
and aided by smoothing properties related to the Neumann heat semigroup $(e^{t \Delta})_{t \geq 0}$ (see Section 2 of \cite{HorstWink} and Lemma 1.3 of \cite{WinklAggre}), we obtain for some $\lambda_1>0$, $C_S>0$ and $c_2>0$, once bounds $v\leq \lVert  v_0 \lVert_{L^{\infty}(\Omega)}$ on $\bar{\Omega}\times (0,\TM)$ and \eqref{Estim_For_f1}  are considered,
\begin{equation*}
\begin{split}
 \lVert  v (\cdot, t) \lVert_{W^{1,2}(\Omega)} & \leq   \lVert e^{t \Delta}v_0 \lVert_{W^{1,2}(\Omega)}+\int_0^t \lVert e^{(t-s)\Delta}f(u(\cdot,s))v(\cdot,s)\lVert_{W^{1,2}(\Omega)}ds \\ &
\leq C_S \lVert v_0\lVert_{W^{1,2}(\Omega)}+C_S\int_0^t \lVert  (-\Delta +1)^\rho e^{(t-s)\Delta}f(u(\cdot,s))v(\cdot,s)\lVert_{L^2(\Omega)}ds
\\ &
\leq C_S \lVert v_0\lVert_{W^{1,2}(\Omega)}+C_S\lVert  v_0 \lVert_{L^{\infty}(\Omega)}|\Omega|^\frac{1}{2}\int_0^t (t-s)^{-\rho-\frac{n}{2}(\alpha-\frac{1}{2})}e^{-\lambda_1 (t-s)} \lVert f(u(\cdot,s))\lVert_{L^\frac{1}{\alpha}(\Omega)}ds \\ &
\leq c_2\Big(1+\int_0^t (t-s)^{-\rho-\frac{n}{2}(\alpha-\frac{1}{2})}e^{-\lambda_1 (t-s)}ds\Big).
\end{split}
\end{equation*}
By recalling the above position on $\zeta$, we introduce the Gamma function $\Gamma$ inferring $
\int_0^t (t-s)^{-\rho-\frac{n}{2}\big(\alpha-\frac{1}{2}\big)}e^{-\lambda_1 (t-s)}ds\leq \lambda_1^{-\zeta} \Gamma(\zeta)$, so obtain the second bound in \eqref{Cg} with $c_0= \max\{\frac{c_1}{2},\int_\Omega \lvert \nabla v_0\rvert^2,c_2^2(1+\lambda_1^{-\zeta} \Gamma(\zeta))^2\}.$
\end{proof}
\end{lemma}
In view of the forthcoming lemma,  in order to ensure the uniform-in-time $L^\infty$ bound of $(u,v,w)$, it will be sufficient in the sequel controlling the uniform-in-time $L^p$-norm of $u$, for some suitable $p>1.$ 
\begin{lemma}\label{FromLocalToGLobalBoundedLemma}
Under the hypotheses of Lemma \ref{LocalExistenceLemma} and any $\xi>0$, let  $(u,v,w)$  be the local-in-time classical solution  to problem \eqref{problem}. If for some $p>\max\{1,\frac{n}{2}\}$ the $u$-component and $g$ belong to $L^\infty((0,T_{max});L^p(\Omega))$, then $(u,v,w)$ is global in time, i.e. $T_{max}=\infty$, and moreover $u, v$ and $w$ are uniformly bounded in $\Omega \times (0,\infty)$ (in the sense of Remark \ref{RemarkUnifBounCalsSol}).
\begin{proof}
W.l.o.g., we assume $p>1$, for $n=1$, and $\frac{n}{2}<p<n$, for $n\geq 2$. In this way, classical regularity theory on elliptic equations in conjunction with Sobolev embedding theorems infer through the third equation of \eqref{problem} that $$w\in L^\infty((0,T_{max});W^{2,p}(\Omega))  \textrm{ and }  \nabla w\in  L^\infty((0,T_{max});W^{1,p}(\Omega)),$$ and so  for all $2\leq n<q<p^*:=\frac{np}{n-p}$, and $q=\infty$ for $n=1$,
\begin{equation}\label{propertiesofW}
 w\in  L^\infty((0,T_{max});C^{[2-(n/p)]}(\bar{\Omega})) \textrm{ and }  \nabla w\in  L^\infty((0,T_{max});L^{q}(\Omega)).
\end{equation} 
On the other hand, the hypotheses on $f$ are such that if $u\in L^\infty((0,\TM);L^p(\Omega))$ also $f\in L^\infty((0,\TM);L^p(\Omega))$. Henceforth, we again use the variation-of-constants formula for $v$ and smoothing properties of the Neumann heat semigroup $(e^{t \Delta})_{t \geq 0}$ as to obtain, taking into account the first bound in \eqref{Cg}, some proper $\tilde{C}_S, c_3>0$ producing for $\lambda_1>0$ as in Lemma \ref{LocalExistenceLemma} 
\begin{equation*}
\begin{split}
 \lVert  \nabla v (\cdot, t) \lVert_{L^{q}(\Omega)} & \leq   \lVert \nabla  e^{t \Delta}v_0 \lVert_{L^{q}(\Omega)}+\int_0^t \lVert \nabla  e^{(t-s)\Delta}f(u(\cdot,s))v(\cdot,s)\lVert_{L^{q}(\Omega)}ds \\ &
\leq \tilde{C}_S \lVert \nabla v_0\lVert_{L^{q}(\Omega)}+\tilde{C}_S\lVert  v_0 \lVert_{L^{\infty}(\Omega)}|\Omega|^\frac{1}{q}\int_0^t (1+(t-s))^{-\frac{1}{2}-\frac{n}{2}(\frac{1}{p}-\frac{1}{q})}e^{-\lambda_1 (t-s)} \lVert f(u(\cdot,s)))\lVert_{L^p(\Omega)}ds \\ &
\leq c_3\Big(1+\int_0^t (t-s)^{-\frac{1}{2}-\frac{n}{2}(\frac{1}{p}-\frac{1}{q})}e^{-\lambda_1 (t-s)}ds\Big).
\end{split}
\end{equation*}
Further, the assumptions on $q$ ensures that $-\frac{1}{2}-\frac{n}{2}(\frac{1}{p}-\frac{1}{q})>-1$, so as before  $\int_0^t (t-s)^{-\frac{1}{2}-\frac{n}{2}(\frac{1}{p}-\frac{1}{q})}e^{-\lambda_1 (t-s)}ds$ is finite and we also get $v \in  L^\infty((0,T_{max});W^{1,q}(\Omega)).$ From this inclusion and \eqref{propertiesofW}, since $W^{1,q}(\Omega) \hookrightarrow L^\infty(\Omega)$ for $q>n$, we immediately have that $v,w\in L^\infty((0,T_{max});L^{\infty}(\Omega))$, and moreover for $\tilde{v}=\chi  v-\xi  w$ some positive constant $C_q$ can be found so to get
\begin{equation}\label{Bound_v_1-q}
\lVert  \tilde{v} (\cdot, t)\lVert_{L^{q}(\Omega)}+\lVert  \nabla \tilde{v} (\cdot, t)\lVert_{L^{q}(\Omega)}  \leq C_q\quad \textrm{for all}\quad t\in(0,T_{max}).
\end{equation}
Subsequently, for any $(x,t)\in \Omega \times (0,T_{max})$, the first equation of  \eqref{problem} reads $u_t=\Delta u-\nabla \cdot (u \nabla \tilde{v})$ and for $t_0:=\max\{0,t-1\}$ we have
\begin{equation*}
\begin{split}
u (\cdot,t) &\leq e^{(t-t_0)\Delta}u(\cdot,t_0)-\int_{t_0}^t e^{(t-s)\Delta}\nabla \cdot (u (\cdot,s) \nabla \tilde{v} (\cdot,s))ds=: u_{1}(\cdot,t)+u_{2}(\cdot,t). 
\end{split}
\end{equation*}
As to the conclusion $u\in L^\infty((0,\TM);L^\infty(\Omega))$, this is an adaptation of \cite[Lemma 3.2]{BellomoEtAl}, and we herewith omit it; more precisely (see also \cite[Lemma 4.1]{ViglialoroWoolleyAplAnal}), the $L^\infty(\Omega)$-norm of $u$ on $(0,T_{max})$ is achieved by  controlling (also with the support of $u\in L^\infty((0,T_{max});L^{p}(\Omega))$, for $\frac{n}{2}<p<n$ only, \eqref{Bound_v_1-q} and \eqref{massConservation}) a suitable  norm of the cross-diffusion term $u \nabla \tilde{v}$. Finally, $u\in L^\infty((0,\TM);L^\infty(\Omega))$ and $v\in L^\infty((0,\TM);W^{1,q}(\Omega))$ imply from the dichotomy criterion \eqref{ExtensibilityCrit} that necessarily we must have $\TM=\infty$, so that actually $u,v,w\in L^\infty((0,\infty);L^\infty(\Omega))$.
\end{proof}
\end{lemma}
\section{A priori estimates and proof of the  theorems}\label{EstimatesAndProofSection}
In this section we control the $L^p$-norm, $p>1$, by establishing an absorptive differential inequality for the functional $y(t):=\int_\Omega u^p +(\frac{\chi^2}{\gamma})^p \int_\Omega |\nabla v|^{2p}$.
\begin{lemma}\label{Estim_general_For_u^pLemma} 
Let $n\geq 1$, $l\geq 1$ and the hypotheses of Lemma \ref{LocalExistenceLemma} be satisfied. Then for every $\xi>0$ the local solution $(u,v,w)$ to problem \eqref{problem} is such that for any $p \in (\max\{l,\frac{l(nl-2)}{n}\},\infty)$ and all $t \in (0,T_{max})$ one has:
\begin{itemize}
\item [$\bullet$] For $l=1$ and some $c_4>0$ 
\begin{equation*}
\frac{d}{dt} \int_\Omega u^p \leq -\frac{2(p-1)}{p}\int_\Omega |\nabla u^{\frac{p}{2}}|^2 
+ \frac{\chi^2 p(p-1)}{2(p+1)} \left(\frac{\xi \gamma (p+1)}{2p^2\chi^2}\right)^{-p} \int_\Omega |\nabla v|^{2(p+1)}
- \frac{\xi \gamma(p-1)}{4} \int_\Omega u^{p+1} + c_4; 
\end{equation*}
\item [$\bullet$]  For $l>1$, every $\epsilon_1, \epsilon_2>0$ and some $c_5>0$
\begin{equation*}
\frac{d}{dt} \int_\Omega u^p \leq \left[-\frac{2(p-1)}{p} + \epsilon_1\right] \int_\Omega |\nabla u^\frac{p}{2}|^2 + \epsilon_2 \int_\Omega |\nabla v|^{2 (p+1)} -\frac{\xi \gamma(p-1)}{4} \int_\Omega u^{p+l} + c_5.
\end{equation*}
\end{itemize}
\begin{proof}
Testing the first equation of problem \eqref{problem} by $u^p$, using its boundary conditions and recalling the properties of $g$ in \eqref{f}, provide on  $(0,T_{max})$
\begin{equation}\label{Estim_1_For_u^p}  
\begin{split}
\frac{1}{p}\frac{d}{dt} \int_\Omega u^p &=\int_\Omega u^{p-1}u_t= -(p-1) \int_\Omega u^{p-2} |\nabla u|^2 + (p-1) \chi \int_\Omega u^{p-1} \nabla u \cdot \nabla v - \xi(p-1) \int_\Omega u^{p-1} 
\nabla u \cdot \nabla w\\
&= -(p-1) \int_\Omega u^{p-2} |\nabla u|^2 + (p-1) \chi \int_\Omega u^{p-1} \nabla u \cdot \nabla v +\frac{\xi \delta (p-1)}{p} \int_\Omega u^p w- \frac{\xi \gamma (p-1)}{p} \int_\Omega u^{p+l},
\end{split}
\end{equation}
whereas Young's inequality infers
\begin{equation} \label{Young_0}  
(p-1)\chi \int_\Omega u^{p-1} \nabla u \cdot \nabla v \leq  \frac{(p-1)}{2}\int_\Omega u^{p-2} |\nabla u|^2  + \frac{\chi^2(p-1)}{2}\int_\Omega u^p |\nabla v|^2\quad \text{on }\, (0,\TM).
\end{equation}
Now, let us analyze separately the two cases.
\begin{itemize}
\item [$\bullet$] Case $l=1$.  The Young inequality entails 
\begin{equation} \label{Young_1}  
\frac{\chi^2(p-1)}{2}\int_\Omega u^p |\nabla v|^2
\leq \frac{\xi \gamma (p-1)}{4p} \int_\Omega u^{p+1}  
+\frac{\chi^2 (p-1)}{2(p+1)} \left(\frac{\xi \gamma (p+1)}{2p^2\chi^2}\right)^{-p} \int_\Omega |\nabla v|^{2(p+1)} \quad \text{ for all } \, t\in (0, T_{max}).
\end{equation}
On the other hand, since $l=1$, an integration over $\Omega$ of the third equation of \eqref{problem}, together with  the mass conservation property \eqref{massConservation}, infer $\int_{\Omega} w=\frac{m\gamma}{\delta}$ on  $(0, T_{max}).$ In this way, by exploiting Young's inequality, again, and relation \eqref{EhrlingTypeInequalityWithMass} with $\psi=w$, $g(u)=\gamma u$ and $\overline{p}=p$, give for suitable positive constants $\hat{c}, c_6$
\begin{equation} \label{Young_2}
\begin{split}
\frac{\xi \delta (p-1)}{p} \int_\Omega u^p w &\leq  \frac{\xi \gamma (p-1)}{4p} \int_\Omega u^{p+1} +\hat{c} \int_\Omega w^{p+1}\leq \frac{\xi \gamma (p-1)}{4p} \int_\Omega u^{p+1} + \frac{\xi \gamma (p-1)}{4p} \int_\Omega u^{p+1} + c_6\quad \text{ on } \; (0, T_{max}).
\end{split}
\end{equation}
By plugging estimates \eqref{Young_0}, \eqref{Young_1} and \eqref{Young_2} into bound \eqref{Estim_1_For_u^p}, and for $c_4=pc_6$, we directly obtain the claim in view of the identity
\[
\int_\Omega u^{p-2} |\nabla u|^2 = \frac{4}{p^2} \int_\Omega |\nabla u^{\frac{p}{2}}|^2 \quad \text{on  } (0,T_{max}).
\]
\item [$\bullet$] Case $l>1$.  Let us first estimate the term $\left(\int_{\Omega} u^l\right)^{\frac{p+l}{l}}$: by applying the Gagliardo--Nirenberg inequality 
(see \cite{Nirenber_GagNir_Ineque}) combined with \eqref{InequalityA+BToPowerP}, for any $\overline{c}>0$ we can introduce a suitable constant $c_7>0$ and obtain 
\begin{equation*}
\overline{c} \left(\int_{\Omega} u^l \right)^{\frac{p+l}{l}}= \overline{c} \|u^{\frac{p}{2}}\|_{L^{\frac{2l}{p}}(\Omega)}^{\frac{2(p+l)}{p}} \leq c_7 \|\nabla u^{\frac{p}{2}}\|_{L^2(\Omega)}^{\frac{2(p+l)}{p}\theta_1} \|u^{\frac{p}{2}}\|_{L^{\frac{2}{p}}(\Omega)}^{\frac{2(p+l)}{p}(1-\theta_1)} + c_7 \|u^{\frac{p}{2}}\|_{L^{\frac{2}{p}}(\Omega)}^{\frac{2(p+l)}{p}} \quad \text{ for all } t \in(0,T_{max}),
\end{equation*}
where for $p$ as in our assumptions we have
\[
\theta_1=\frac{1-\frac{1}{l}}{1+\frac{2}{np}-\frac{1}{p}} \in (0,1).
\]
Hence, by recalling the mass conservation property \eqref{massConservation}, and in view of $\frac{(p+l)}{p}\theta_1 <1$, the Young and above inequalities entail for  $c_8,c_9>0$ and any $\epsilon_1>0$ 
\begin{equation}\label{GN3}
\overline{c} \left(\int_\Omega u^l \right)^{\frac{p+l}{l}} \leq c_8 \Big(\int_\Omega |\nabla u^\frac{p}{2}|^2\Big)^{\frac{(p+l)}{p}\theta_1}+ c_8 \leq \frac{\epsilon_1}{p} \int_\Omega |\nabla u^\frac{p}{2}|^2  + c_9 \quad \text{with } t \in (0,T_{max}).
\end{equation}
On the other hand,  by noting that $\frac{2(p+l)}{l} < 2(p+1)$ for $l>1$, a double application of the Young inequality leads for all $t\in (0,T_{max})$, $\epsilon_2>0$ and  some $c_{10}, c_{11} >0$ to
\begin{equation}\label{Young_4}
\frac{\chi^2(p-1)}{2}\int_\Omega u^p |\nabla v|^2 \leq \frac{\xi \gamma (p-1)}{4p} \int_\Omega u^{p+l} + c_{10} \int_\Omega |\nabla v|^{2 \frac{p+l}{l}}
\leq \frac{\xi \gamma (p-1)}{4p} \int_\Omega u^{p+l} +\frac{\epsilon_2}{p}  \int_\Omega |\nabla v|^{2 (p+1)} + c_{11}.
\end{equation} 
Now by applying restrictions in \eqref{f} and relation \eqref{InequalityA+BToPowerP}, as well as the obvious inequality $u\leq u+1$, we have these estimates for any $\bar{p}>1$ and some $c_{12},c_{13},c_{14},c_{15}>0$:
\begin{equation*}
\int_{\Omega} (g(u))^{\overline{p}+1}\leq \int_{\Omega} \left(\gamma u (u+1)^{l-1}\right)^{\overline{p}+1} 
\leq c_{12} \int_{\Omega} u^{l(\overline{p}+1)} +c_{13}\quad \text{for all } t\in (0,\TM),
\end{equation*}
and also
\begin{equation*}
\left(\int_{\Omega} g(u)\right)^{\overline{p}+1} \leq c_{14} \left(\int_{\Omega} u^l \right)^{\overline{p}+1}  +c_{15} \quad \text{for all } t\in (0,\TM),
\end{equation*}
Aided by the gained estimates, we now use a combination of Young's inequality and relation \eqref{EhrlingTypeInequalityWithMass} with $\psi=w$, $\overline{p}=\frac{p}{l}>1$; we get with some $c_{16}>0$
\begin{equation} \label{Young_5}
\begin{split}
\frac{\xi \delta (p-1)}{p} \int_\Omega u^p w &\leq \frac{\xi \gamma (p-1)}{4p} \int_\Omega u^{p+l} + \frac{\hat{c}}{c_{12}}\int_\Omega w^{\overline{p}+1}\\
&\leq \frac{\xi \gamma (p-1)}{4p} \int_\Omega u^{p+l} + \sigma \int_\Omega u^{p+l} + \overline{c} \left(\int_{\Omega} u^l \right)^{\frac{p+l}{l}} + c_{16}\quad \text{on }\, (0,\TM).
\end{split}
\end{equation}
By collecting \eqref{Young_0}, \eqref{Young_4}, and \eqref{Young_5}, with $\sigma=\frac{\xi\gamma(p-1)}{4p}$, bound \eqref{Estim_1_For_u^p} gives the conclusions also in view of relation \eqref{GN3}.
\end{itemize}
\end{proof}
\end{lemma}
In the forthcoming lemma we adapt to our framework some derivations already developed in  \cite[Lemma 4.2]{LankeitWangConsumptLogistic}.
\begin{lemma}\label{Estim_general_For_nablav^2pLemma} 
Let $n\geq 1$, $l\geq 1$ and the hypotheses of Lemma \ref{LocalExistenceLemma} be satisfied. Then for every $\xi>0$ the local solution $(u,v,w)$ to problem \eqref{problem} is such for any $p \in (1,\infty)$ and all $t\in (0,\TM)$ one has: 
\begin{itemize}
\item [$\bullet$] For $l=1$ and some $c_{17}>0$ 
\begin{equation*}
\begin{split}
\left(\frac{\chi^2}{\gamma}\right)^p\frac{d}{dt}\int_\Omega  |\nabla v|^{2p} + \left(\frac{\chi^2}{\gamma}\right)^pp \int_\Omega |\nabla v|^{2p-2} |D^2v|^2& \leq \frac{\xi\gamma(p-1)}{4} \int_\Omega u^{p+1} 
\\ & 
\quad + \frac{p}{8(4p^2+n)\|v\|_{L^{\infty}(\Omega)}^2}\left(\frac{\chi^2}{\gamma}\right)^p \int_\Omega |\nabla v|^{2(p+1)} + c_{17};
\end{split}
\end{equation*}
\item [$\bullet$] For $l>1$, every positive  $\epsilon_3$ and some $c_{18}>0$ 
\begin{equation*}
\left(\frac{\chi^2}{\gamma}\right)^p\frac{d}{dt}\int_\Omega  |\nabla v|^{2p}+ \left(\frac{\chi^2}{\gamma}\right)^pp \int_\Omega |\nabla v|^{2p-2} |D^2v|^2 \leq \frac{\xi \gamma(p-1)}{4} \int_\Omega u^{p+l} + \epsilon_3 \int_\Omega |\nabla v|^{2(p+1)} + c_{18}.
\end{equation*}
\end{itemize}
\begin{proof}
From the second equation of \eqref{problem}, we derive this pointwise identity valid for all $x\in \Omega$ and $t\in(0,T_{max})$:
\begin{equation*}
(|\nabla v|^2)_t =2 \nabla v\cdot \nabla v_t=2 \nabla v \cdot \nabla \Delta v -2\nabla v \cdot \nabla(f(u) v)
= \Delta |\nabla v|^2 -2 |D^2 v|^2  -2\nabla v \cdot \nabla (f(u) v).
\end{equation*}
Successively, multiplying this last relation by $\lvert \nabla v\rvert^{2p-2}$ and integrating over $\Omega$ lead to
\begin{equation}\label{A_00}
\begin{split}
&\left(\frac{\chi^2}{\gamma}\right)^p \frac{1}{p}\frac{d}{dt}\int_\Omega  |\nabla v|^{2p}+(p-1)\left(\frac{\chi^2}{\gamma}\right)^p \int_\Omega |\nabla v|^{2p-4}|\nabla |\nabla v|^2|^2+2\left(\frac{\chi^2}{\gamma}\right)^p \int_\Omega |\nabla v|^{2p-2} |D^2v|^2\\ &
\quad = -2 \left(\frac{\chi^2}{\gamma}\right)^p \int_\Omega |\nabla v|^{2p-2} \nabla v \cdot \nabla (f(u) v) +\left(\frac{\chi^2}{\gamma}\right)^p \int_{\partial \Omega} (|\nabla v|^2)^{p-1} (|\nabla v|^2)_{\nu} 
\quad \textrm{for all}\quad t \in (0,T_{max}).
\end{split}
\end{equation}
Now, by virtue of the bound for $\nabla v$ in \eqref{Cg}, we apply estimate \eqref{NoConvexity} with $\psi=v$ so to obtain for $c_{19}=C_{\eta}c_0^p$ 
\begin{equation}\label{NoConv}
\left(\frac{\chi^2}{\gamma}\right)^p \int_{\partial \Omega} (|\nabla v|^2)^{p-1} (|\nabla v|^2)_{\nu} \leq \eta \left(\frac{\chi^2}{\gamma}\right)^p \int_\Omega |\nabla v|^{2p-4}|\nabla |\nabla v|^2|^2 + c_{19} 	\quad \text{on } (0,T_{max}).
\end{equation}
Hence, an integration by parts to the right hand side term of \eqref{A_00} produces, also thanks to the first estimate in \eqref{Cg} and assumption \eqref{f}, 
\begin{equation}\label{A_01} 
\begin{split}
-2 \left(\frac{\chi^2}{\gamma}\right)^p &\int_\Omega |\nabla v|^{2p-2} \nabla v \cdot \nabla (f(u) v) = 2 \left(\frac{\chi^2}{\gamma}\right)^p 
\int_\Omega f(u) v|\nabla v|^{2p-2}\Delta v\\ 
&\quad+ 2(p-1) \left(\frac{\chi^2}{\gamma}\right)^p \int_\Omega f(u) v|\nabla v|^{2p-4}\nabla v \cdot \nabla \lvert\nabla v\rvert^2 \\& 
 \leq 2 K \left(\frac{\chi^2}{\gamma}\right)^p \lVert v_0\rVert_{L^\infty(\Omega)}\int_\Omega u^{\alpha}|\nabla v|^{2p-2}|\Delta v|\\ 
&+2 K \left(\frac{\chi^2}{\gamma}\right)^p (p-1) \lVert v_0\rVert_{L^\infty(\Omega)}\int_\Omega u^{\alpha}|\nabla v|^{2p-3}\lvert\nabla \lvert \nabla v\rvert^2\rvert 
\quad \textrm{for all}\quad t \in (0,T_{max}).
\end{split}
\end{equation}
In addition, the Young and \eqref{InequalityLaplacian} inequalities allow us to derive for some $c_{20},c_{21}>0$
\begin{equation}\label{Estimateu^2gradv^2_1}
2 K \left(\frac{\chi^2}{\gamma}\right)^p \lVert v_0\rVert_{L^\infty(\Omega)}\int_\Omega u^{\alpha}|\nabla v|^{2p-2}|\Delta v| \leq  \left(\frac{\chi^2}{\gamma}\right)^p \int_\Omega |\nabla v|^{2p-2}|D^2 v|^2 
+ c_{20} \int_\Omega u^{2\alpha} |\nabla v|^{2p-2} \quad \textrm{on } \;  (0,T_{max})
\end{equation}
and similarly for all $t \in (0,T_{max})$ and any $\eta \in (0,p-1)$
\begin{equation}\label{Estimateu^2gradv^2_Second}
 2 K \left(\frac{\chi^2}{\gamma}\right)^p (p-1) \lVert v_0\rVert_{L^\infty(\Omega)}\int_\Omega u^{\alpha} |\nabla v|^{2p-3}\lvert\nabla \lvert \nabla v\rvert^2\rvert\leq (p-1-\eta)\left(\frac{\chi^2}{\gamma}\right)^p \int_\Omega |\nabla v|^{2p-4}\lvert \nabla \lvert \nabla v\rvert^2\rvert^2 + c_{21} \int_\Omega u^{2\alpha} |\nabla v|^{2p-2}.
\end{equation}
By inserting \eqref{NoConv}, \eqref{A_01}, \eqref{Estimateu^2gradv^2_1} and \eqref{Estimateu^2gradv^2_Second} into \eqref{A_00}, we deduce that for some 
$c_{22},c_{23}>0$
\begin{equation}\label{Estim}
\left(\frac{\chi^2}{\gamma}\right)^p \frac{1}{p}\frac{d}{dt}\int_\Omega  |\nabla v|^{2p}+ \left(\frac{\chi^2}{\gamma}\right)^p \int_\Omega |\nabla v|^{2p-2} |D^2v|^2 \leq c_{22}
\int_\Omega u^{2\alpha} |\nabla v|^{2p-2} +c_{23} \quad \textrm{on } \;  (0,T_{max}).
\end{equation}
Further, for all $l\geq 1$, the Young inequality also gives on $(0,T_{max})$ 
\begin{equation}\label{Estimating1nablav^2p+2}
c_{22} \int_\Omega u^{2\alpha} |\nabla v|^{2p-2} 
\leq \frac{\xi \gamma (p-1)}{4p} \int_\Omega u^{p+l} + c_{24} \int_\Omega |\nabla v|^{\frac{2(p-1)(p+l)}{p+l-2\alpha}}
\leq \frac{\xi \gamma (p-1)}{4p} \int_\Omega u^{p+l} + \frac{\epsilon_3}{p} \int_\Omega |\nabla v|^{2(p+1)} + c_{25},
\end{equation}
where we have used that $\frac{2(p-1)(p+l)}{p+l-2\alpha} < 2(p+1)$ (recall $0<\alpha<1$) and $\epsilon_3$ is an arbitrarily positive constant and $c_{24}, c_{25}>0$. We have the two claims introducing \eqref{Estimating1nablav^2p+2} into \eqref{Estim}, with an evident choice of $\epsilon_3$ when $l=1.$ 
\end{proof}
\end{lemma}
\begin{lemma}\label{LemmaAbsorptiveMainInequality}
	Let $n\geq 1$, $l=1$ and  the hypotheses of Lemma \ref{LocalExistenceLemma} be satisfied. Then, for any $p \in (1, \infty)$ there exists $\tilde{C}(p,n)\geq 0$ such that for all $\xi>0$ fulfilling 
	\begin{equation}\label{AssumptionOnXi}
	\xi > \left(\frac{4 \tilde{C}(p,n) \|\chi v_0\|_{L^{\infty}(\Omega)}^2}{p}\right)^{\frac{1}{p}}, 
	\end{equation}
	the following holds true: For some $L>0$ the $u$-component of the local solution $(u,v,w)$ to problem \eqref{problem} complies with 
	\begin{equation*}
	\int_\Omega u^p \leq L \quad \textrm{for all}\quad t \in (0,T_{max}).
	\end{equation*}
	Additionally, the same conclusion is valid whenever $n\geq 1$, $l>1$, $\xi>0$ and all $p\in (\max\{l,\frac{l(nl-2)}{n}\},\infty).$ 
\begin{proof}
When $l=1$, Lemma \ref{Estim_general_For_u^pLemma}, Lemma \ref{Estim_general_For_nablav^2pLemma}  and relation \eqref{InequalityHessian}, supported by the bound for $v$ in \eqref{Cg}, imply that
\begin{equation}\label{Somma}
\begin{split}
&\frac{d}{dt} \left(\int_\Omega u^p + \left(\frac{\chi^2}{\gamma}\right)^p \int_\Omega  |\nabla v|^{2p}\right) + \frac{2(p-1)}{p}\int_\Omega |\nabla u^{\frac{p}{2}}|^2 
+p \left(\frac{\chi^2}{\gamma}\right)^p \int_\Omega |\nabla v|^{2p-2} |D^2v|^2\\
&\leq \left(\frac{p}{8(4p^2+n)\|v\|_{L^{\infty}(\Omega)}^2}\left(\frac{\chi^2}{\gamma}\right)^p+ \frac{\chi^2 p(p-1)}{2(p+1)} \left(\frac{\xi \gamma (p+1)}{2p^2\chi^2}\right)^{-p}\right) 
\int_\Omega |\nabla v|^{2(p+1)}+ c_{26}\\
& \leq \left(\frac{p}{8(4p^2+n)\|v\|_{L^{\infty}(\Omega)}^2} \left(\frac{\chi^2}{\gamma}\right)^p + \frac{\chi^2 p(p-1)}{2(p+1)} \left(\frac{\xi \gamma (p+1)}{2p^2\chi^2}\right)^{-p}\right) 
2 (4p^2+n)\lVert v_0\Vert^2_{L^\infty(\Omega)}\int_\Omega \lvert \nabla  v\rvert^{2p-2} \lvert D^2 v\rvert^2 + c_{26}\\
&= \left(\frac{\chi^2}{\gamma}\right)^p \left(\frac{p}{4} +  \frac{\tilde{C}(p,n)}{\xi^p} \|\chi v_0\|_{L^{\infty}(\Omega)}^2 \right)
\int_\Omega \lvert \nabla  v\rvert^{2p-2} \lvert D^2 v\rvert^2 + c_{26} \quad \text{ on } (0, T_{max}),
\end{split}
\end{equation}
where $c_{26} >0$ and 
$$
\tilde{C}(p,n)=
\begin{cases}
0 & n\in \{1,2\},\\
\frac{2^{p} p^{2p+1}(p-1)(4p^2+n)}{(p+1)^{p+1}}& n\geq 3. 
\end{cases}
$$
 Since by our assumptions $\xi$ satisfies restriction \eqref{AssumptionOnXi}, we get that 
$\frac{\tilde{C}(p,n)}{\xi^p} \|\chi v_0\|_{L^{\infty}(\Omega)}^2 < \frac{p}{4}$; henceforth relation \eqref{Somma} actually reads 
\[
\frac{d}{dt} \left(\int_\Omega u^p + \left(\frac{\chi^2}{\gamma}\right)^p \int_\Omega  |\nabla v|^{2p}\right) + \frac{2(p-1)}{p}\int_\Omega |\nabla u^{\frac{p}{2}}|^2 
+ \frac{p}{2} \left(\frac{\chi^2}{\gamma}\right)^p  \int_\Omega |\nabla v|^{2p-2} |D^2v|^2 \leq c_{26}\quad \textrm{on } (0,\TM).
\]
On the other hand, from inequality \eqref{InequalityGradienHessian} we have
\[\vert \nabla \lvert \nabla v\rvert^p\rvert^2=\frac{p^2}{4}\lvert \nabla v \rvert^{2p-4}\vert \nabla \lvert \nabla v\rvert^2\rvert^2=p^2\lvert \nabla v \rvert^{2p-4}\lvert D^2v \nabla v \rvert^2\leq p^2|\nabla v|^{2p-2} |D^2v|^2,\]
so that we obtain 
\begin{equation}\label{Estim_general_For_y_2}
y'(t)+\frac{2(p-1)}{p}\int_\Omega \lvert \nabla u^\frac{p}{2}\rvert^2+\frac{1}{2p} \left(\frac{\chi^2}{\gamma}\right)^p \int_\Omega \vert \nabla \lvert \nabla v\rvert^p\rvert^2 \leq c_{26}
\quad \text{ on } (0, T_{max}).
\end{equation}
Conversely, for $l>1$, by relying again on  Lemma \ref{Estim_general_For_u^pLemma}, Lemma \ref{Estim_general_For_nablav^2pLemma}, any $\epsilon>0$ and some $c_{27}>0$ entail 
\begin{equation*}
\begin{split}
\frac{d}{dt} \left(\int_\Omega u^p + \left(\frac{\chi^2}{\gamma}\right)^p \int_\Omega  |\nabla v|^{2p}\right) & +\left[\frac{2(p-1)}{p}- \epsilon_1\right]\int_\Omega |\nabla u^{\frac{p}{2}}|^2 + p \left(\frac{\chi^2}{\gamma}\right)^p \int_\Omega |\nabla v|^{2p-2} |D^2v|^2
\leq  \\ & 
\frac{\epsilon p (\frac{\chi^2}{\gamma})^p}{2 (4p^2+n)\lVert v_0\Vert^2_{L^\infty(\Omega)}} \int_\Omega |\nabla v|^{2 (p+1)}+ c_{27}\quad \text{for all } t\in (0,\TM),
\end{split}
\end{equation*}
which similarly to what has been previously done, by choosing $\epsilon_1 \in \left(0, \frac{2(p-1)}{p}\right)$ and $\epsilon\in \left(0, 1\right)$, we obtain an absorptive inequality similar to \eqref{Estim_general_For_y_2}; then, both can be unified  for suitable positive constants $a,b,c$ as
\begin{equation}\label{Estim_general_For_y_2_general}
y'(t)+a\int_\Omega \lvert \nabla u^\frac{p}{2}\rvert^2+b \int_\Omega \vert \nabla \lvert \nabla v\rvert^p\rvert^2\leq  c
\quad \text{ on } (0, T_{max}).
\end{equation}
Successively, for any $l\geq 1,$ by exploiting again the Gagliardo--Nirenberg inequality, there exists a positive constant $c_{28}$ such that
\begin{equation*} 
\int_\Omega u^{p}=\lvert \lvert u^\frac{p}{2}\lvert \lvert_{L^2(\Omega)}^2 \leq c_{28}  \lvert \lvert\nabla u^\frac{p}{2}\lvert \lvert_{L^2(\Omega)}^{2 \theta} \lvert \lvert u^\frac{p}{2}\lvert \lvert_{L^\frac{2}{p}(\Omega)}^{2(1-\theta)} +c_{28} \lvert \lvert u^\frac{p}{2}\lvert \lvert^2_{L^\frac{2}{p}(\Omega)} \quad \text{ for all } t \in(0,T_{max}),
 \end{equation*}
 with
\[0<\theta=\frac{\frac{np}{2}(1-\frac{1}{p})}{1-\frac{n}{2}+\frac{np}{2}}<1.\]
Taking into consideration bound \eqref{massConservation} and introducing $c_{29}>0$, the two above inequalities lead to 
\begin{equation}\label{Estim_utopGaglNiren}
\int_\Omega u^{p}\leq c_{29}\Big(\int_\Omega \lvert \nabla u^\frac{p}{2}\rvert^2\Big)^{\theta}+c_{29} \quad \textrm{on } \, (0,T_{max}).
\end{equation}
In a similar way, another application of the Gagliardo--Nirenberg produces some  $c_{30}>0$ such that 
\begin{equation*}
\int_\Omega \lvert \nabla v\rvert^{2p}=\lvert \lvert \lvert \nabla v\rvert^p\lvert \lvert_{L^2(\Omega)}^2 
\leq c_{30} \lvert \lvert\nabla  \lvert \nabla v \rvert^p\rvert \lvert_{L^2(\Omega)}^{2\theta} \lvert \lvert\lvert \nabla v \rvert^p\lvert \lvert_{L^\frac{2}{p}(\Omega)}^{2(1-\theta)} +c_{30}      \lvert \lvert \lvert \nabla v \rvert^p\lvert \lvert^2_{L^\frac{2}{p}(\Omega)}\quad \textrm{with } t\in (0,\TM).
 \end{equation*}
Successively, by relying on the bound for $\nabla v$ in \eqref{Cg}, we have for $c_{31}>0$
\begin{equation}\label{Estim_Nabla nabla v^p^2}
\int_\Omega \lvert \nabla v\rvert^{2p}\leq c_{31} \Big(\int_\Omega \lvert \nabla \lvert \nabla v \rvert^p\rvert^2\Big)^{\theta}+c_{31} \quad \textrm{with } t \in (0,T_{max}).
\end{equation}
As a consequence of all of the above, by manipulating inequalities \eqref{Estim_utopGaglNiren} and \eqref{Estim_Nabla nabla v^p^2} and successively using the results into \eqref{Estim_general_For_y_2_general}, we can observe also by virtue of \eqref{InequalityForFinallConclusion} that $y$ satisfies this initial problem
\begin{equation*}\label{MainInitialProblemWithM}
\begin{cases}
y'(t)\leq c_{32}-c_{33} y^{\frac{1}{\theta}}(t)\quad \textrm{for all } t \in (0,T_{max}),\\
y(0)=\int_\Omega u_0^p+ (\frac{\chi^2}{\gamma})^p \int_\Omega |\nabla v_0|^{2p}, 
\end{cases}
\end{equation*}
with $c_{32}, c_{33}$ positive constants. Consequently, an ODE comparison principle implies that 
$\int_\Omega u^p\leq y(t)\leq \max\{y(0),\big(\frac{c_{32}}{c_{33}}\big)^{\theta}\}:=L$ for all $t\in(0,T_{max})$.
\end{proof}
\end{lemma}
Now we have all the necessary tools to conclude.
\subsubsection*{{\bf{Proof of Theorems \ref{MainTheorem} and \ref{Main1Theorem}}}} 
For $l=1$, let $\tilde{C}(p,n)$ be the constant defined in Lemma \ref{LemmaAbsorptiveMainInequality} and let us set 
\[
C(n)=\begin{cases}
0 & \text{if } n\in\{1,2\},\\
(\frac{8}{n}\tilde{C}(n/2,n))^\frac{2}{n} & \text{if } n\geq 3.
\end{cases}
\]	
From our hypotheses, $\xi>{C}(n)\|\chi  v_0\|_{L^{\infty}(\Omega)}^\frac{4}{n}$, so that from continuity arguments we can always pick  $p>\max\{1,\frac{n}{2}\}$ such that assumption \eqref{AssumptionOnXi} holds true. Henceforth, Lemma \ref{LemmaAbsorptiveMainInequality} ensures that the $u$-component of the local solution $(u,v,w)$ to problem \eqref{problem} belongs to $L^{\infty}((0,T_{max}); L^p(\Omega))$; since $l=1$ also $g \in L^{\infty}((0,T_{max}); L^p(\Omega))$ and
the claim follows by invoking Lemma \ref{FromLocalToGLobalBoundedLemma}. Indeed, for any $l>1$, upon enlarging $p$ in the same Lemma  \ref{LemmaAbsorptiveMainInequality}, we also can have $u,g \in L^{\infty}((0,T_{max}); L^p(\Omega))$ for $p>\max\{1,\frac{n}{2}\}$, and identically conclude.
\qed
\section{Logistic source vs. chemorepellent in chemotaxis-consumption models: Which one is more effective toward boundedness? }\label{SectionComparison}
We complement this research by discussing some differences and analogies between a chemotaxis-consumption model with logistic source and that presented here with chemorepellent (linearly produced). To be precise, when the equation for $u$ in problem \eqref{problem} is replaced by \eqref{EquationWithLogistic}, and we set $f(u)=u$ in that for $v$, the chemotaxis-consumption model with logistic source (indicated with $P_\mu$ below) is obtained, and in \cite{LankeitWangConsumptLogistic} boundedness of solutions is established for $\mu$ large with respect to $\chi  \lVert v_0 \rVert_{L^\infty(\Omega)}$. Conversely, for our investigated attraction-repulsion model $P_\xi$ (to facilitate the comparison, we also re-write it next to $P_\mu$) an analogous largeness restriction is moved to the parameter $\xi$:  
\begin{equation*}
P_\mu:\;
\begin{cases}
u_t= \Delta u - \chi \nabla \cdot (u \nabla v)+k u -\mu u^2 \\
v_t=\Delta v-uv 
\end{cases}
\quad \text{and}\quad 
P_\xi:\,
\begin{cases}
u_t= \Delta u - \chi \nabla \cdot (u \nabla v)+\xi \nabla \cdot (u \nabla w) \\
v_t=\Delta v-f(u)v  \\
0= \Delta w - \delta w + \gamma u
\end{cases}
.
\end{equation*}
In particular, some straightforward computations, show that the condition in \cite[Theorem 1.1]{LankeitWangConsumptLogistic} reads
	\begin{equation*}
	\begin{split}
	\mu&>\frac{4^{1/n} (n-1) n}{n+1} \left(\frac{(n-1)(4 n^2+n)}{n+1}\right)^{1/n}\l \Vert \chi   v_0 \rVert_{L^\infty(\Omega)}^\frac{2}{n}\\& 
	\quad +\frac{2^{\frac{n-1}{2}+n+1} (2 n-1)}{n+1} \left(\frac{(n-1) (2 n-1)(4 n^2+n)}{n+1}\right)^{\frac{n-1}{2}}\lVert \chi v_0 \rVert_{L^\infty(\Omega)}^{2n}
	=:C_\mu(\chi \lVert v_0 \rVert_{L^\infty(\Omega)}),
	\end{split}
	\end{equation*}
	whereas in Theorem \ref{MainTheorem} the correlated assumption appears as 
	$$
	\xi>\left(2^{2-\frac{n}{2}} \left(\frac{n}{2}-1\right) \left(\frac{n}{2}+1\right)^{-\frac{n}{2}-1} n^n \left(n^2+n\right)\right)^{2/n} \lVert \chi v_0 \rVert_{L^\infty(\Omega)}^\frac{2}{n}=:C_\xi(\chi \lVert v_0 \rVert_{L^\infty(\Omega)}).
	$$
	Even though a very direct comparison between models $P_\mu$ and $P_\xi$ is not strictly possible, from Figure \ref{ComparisonConditionsLogisticRepulsionFigure} it can be observed that quantitatively  $C_\mu(\chi \lVert v_0 \rVert_{L^\infty(\Omega)})>C_\xi(\chi \lVert v_0 \rVert_{L^\infty(\Omega)})$, for any value of $\chi \lVert v_0 \rVert_{L^\infty(\Omega)}$.
	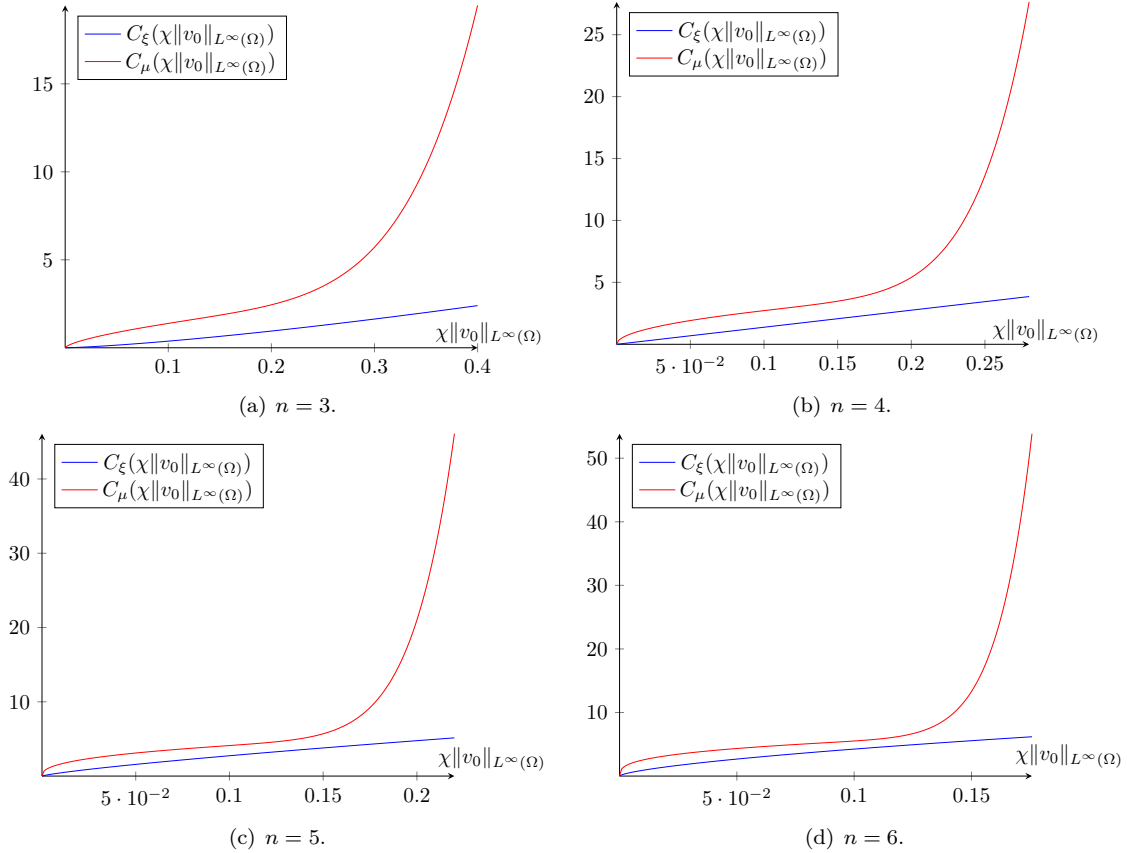
\begin{figure}[h!]  
		\centering  
		\subfigure[$n=3.$]  
		{  
			\begin{tikzpicture}[scale=0.8]
			\begin{axis}[
			axis lines=middle,
			clip=false,
			ymin=0,
			legend pos=north west
			]
			\addplot+[mark=none,samples=200,domain=0:0.4] {36/5 *(6/5)^(2/3)*x^(4/3)};
			\addplot+[mark=none,samples=200,domain=0:0.4] {(3*39^(1/3)*x^(2/3))*2^(-2/3) + 3900*x^6};
			\legend{$C_\xi(\chi \lVert v_0 \rVert_{L^\infty(\Omega)})$};
			\addlegendentry{$C_\mu(\chi \lVert v_0 \rVert_{L^\infty(\Omega)})$};
			\draw[fill] (axis cs:0.352,1.9) node[below right] {$\chi \lVert v_0 \rVert_{L^\infty(\Omega)}$};
			\end{axis}
			\end{tikzpicture}
		}  
		\subfigure[$n=4.$]  
		{  
			\begin{tikzpicture}[scale=0.8]
			\begin{axis}[
			axis lines=middle,
			clip=false,
			ymin=0,
			legend pos=north west
			]
			\addplot+[mark=none,samples=200,domain=0:0.28] {32/3* (5/3)^(1/2)*x};
			\addplot+[mark=none,samples=200,domain=0:0.28] {(24/5)*(51/5)^(1/4)*(x)^(1/2)  + 1279488/25*(714/5)^(1/2)*x^8};
			\legend{$C_\xi(\chi \lVert v_0 \rVert_{L^\infty(\Omega)})$};
			\addlegendentry{$C_\mu(\chi \lVert v_0 \rVert_{L^\infty(\Omega)})$};
			\draw[fill] (axis cs:0.25,2.654) node[below right] {$\chi \lVert v_0 \rVert_{L^\infty(\Omega)}$};
			\end{axis}
			\end{tikzpicture}
			
		}  
		\subfigure[$n=5.$]  
		{  
			
			\begin{tikzpicture}[scale=0.8]
			\begin{axis}[
			axis lines=middle,
			clip=false,
			ymin=0,
			legend pos=north west
			]
			\addplot+[mark=none,samples=200,domain=0:0.22] {50/7*(5/7)^(2/5)*2^(1/5)* 3^(4/5)*x^(4/5)};
			\addplot+[mark=none,samples=200,domain=0:0.22] {10/3*2^(3/5)*35^(1/5)*x^(2/5) + 152409600*x^10};
			\legend{$C_\xi(\chi \lVert v_0 \rVert_{L^\infty(\Omega)})$};
			\addlegendentry{$C_\mu(\chi \lVert v_0 \rVert_{L^\infty(\Omega)})$};
			\draw[fill] (axis cs:0.2074,4.9) node[below right] {$\chi \lVert v_0 \rVert_{L^\infty(\Omega)}$};
			\end{axis}
			\end{tikzpicture}
		}
		\subfigure[$n=6.$]  
		{  
			\begin{tikzpicture}[scale=0.8]
			\begin{axis}[
			axis lines=middle,
			clip=false,
			ymin=0,
			 xtick={0,0.05,...,1},
			legend pos=north west
			]
			\addplot+[mark=none,samples=200,domain=0:0.1758] {9*(21/2)^(1/3)*x^(2/3)};
			\addplot+[mark=none,samples=200,domain=0:0.1758] {30/7*(3/7)^(1/6)*(10)^(1/2) *x^(1/3) + 3833280000000/343*(165/7)^(1/2)*x^12};
			\legend{$C_\xi(\chi \lVert v_0 \rVert_{L^\infty(\Omega)})$};
			\addlegendentry{$C_\mu(\chi \lVert v_0 \rVert_{L^\infty(\Omega)})$};
			\draw[fill] (axis cs:0.1656,6.21) node[below right] {$\chi \lVert v_0 \rVert_{L^\infty(\Omega)}$};
			\end{axis}
			\end{tikzpicture}
		}  
		\caption{Illustration comparing for $n\in\{3,4,5,6\}$ the graphs of the functions $C_\mu(\chi \lVert v_0 \rVert_{L^\infty(\Omega)})$ and $C_\xi(\chi \lVert v_0 \rVert_{L^\infty(\Omega)})$ computed in \cite[Theorem 1.1]{LankeitWangConsumptLogistic} and in Theorem \ref{MainTheorem}, respectively.  }\label{ComparisonConditionsLogisticRepulsionFigure}
	\end{figure}

	For the sake of scientific clarity, we would like to stress that the curve trend of the function $C_\mu$ may be improved; this is essentially due to the fact that in \cite{LankeitWangConsumptLogistic}  the authors prove the deduction ``$L^p\Rightarrow L^\infty$'' for $p>n$, and not for $p>\frac{n}{2}$, as we performed in Lemma \ref{FromLocalToGLobalBoundedLemma}. (As known, in this context, $\frac{n}{2}$ is the smallest value toward the validity of the above implication.)	
In this sense, by adjusting to this choice of $p$ the expression of $C_\mu$, the situation is different. More precisely, for sufficiently large values of $\chi \lVert v_0 \rVert_{L^\infty(\Omega)}$, for which the analysis is more interesting, $C_\mu(\chi \lVert v_0 \rVert_{L^\infty(\Omega)})\gg C_\xi( \chi \lVert v_0 \rVert_{L^\infty(\Omega)})$; on the other hand, the same does not happen when $\chi \lVert v_0 \rVert_{L^\infty(\Omega)}$ is small. (See Figure \ref{ComparisonConditionsLogisticRepulsionFigureN-Mezzi}.) 
\begin{figure}[h!]  
	\centering  
		\subfigure[$n=3.$]  
	{  
		\begin{tikzpicture}[scale=0.8]
		\begin{axis}[
		axis lines=middle,
		clip=false,
		ymin=0,
		 xtick={0,0.13,...,1},
		legend pos=north west
		]
		\addplot+[mark=none,samples=200,domain=0:0.7] {(36/5)*(6/5)^(2/3)* x^(4/3)};
		\addplot+[mark=none,samples=200,domain=0:0.7] {6/5*(6/5)^(2/3)*x^(4/3)+56/5*(21/5)^(1/4)*x^3};
		\legend{$C_\xi(\chi \lVert v_0 \rVert_{L^\infty(\Omega)})$};
		\addlegendentry{$C_\mu(\chi \lVert v_0 \rVert_{L^\infty(\Omega)})$};
		\draw[fill] (axis cs:0.60,0.7) node[below right] {$\chi \lVert v_0 \rVert_{L^\infty(\Omega)}$};
		\draw[fill] (axis cs:{0.598,4.10}) circle [radius=1.7pt] node[above left] {};
		\draw [dashed] (axis cs:0.598,4.00) -- (axis cs:0.598,0);
		\draw[fill] (axis cs:0.555,-0.10) node[below right] {$\rho_0$};
		\end{axis}
		\end{tikzpicture}
	}  
	\subfigure[$n=4.$]  
	{  
		\begin{tikzpicture}[scale=0.8]
		\begin{axis}[
		axis lines=middle,
		clip=false,
		ymin=0,
		legend pos=north west
		]
		\addplot+[mark=none,samples=200,domain=0:0.6] {32/3 *(5/3)^(1/2)* x};
		\addplot+[mark=none,samples=200,domain=0:0.6] {8/3*(5/3)^(1/2)*x + 400/3*(2/3)^(1/2)*x^4};
		\legend{$C_\xi(\chi \lVert v_0 \rVert_{L^\infty(\Omega)})$};
		\addlegendentry{$C_\mu(\chi \lVert v_0 \rVert_{L^\infty(\Omega)})$};
		\draw[fill] (axis cs:0.468,1.9) node[below right] {$\chi \lVert v_0 \rVert_{L^\infty(\Omega)}$};
		\draw[fill] (axis cs:{0.456,6.32}) circle [radius=1.7pt] node[above left] {};
		\draw [dashed] (axis cs:0.456,6.50) -- (axis cs:0.456,0);
		\draw[fill] (axis cs:0.426,-0.3) node[below right] {$\rho_0$};
		\end{axis}
		\end{tikzpicture}
		
	}  
	\subfigure[$n=5.$]  
	{  
		
		\begin{tikzpicture}[scale=0.8]
		\begin{axis}[
		axis lines=middle,
		clip=false,
		ymin=0,
		legend pos=north west
		]
		\addplot+[mark=none,samples=200,domain=0:0.5] {50/7*(5/7)^(2/5)* 2^(1/5)*3^(4/5)*x^(4/5)};
		\addplot+[mark=none,samples=200,domain=0:0.5] {15/7*(5/7)^(2/5)*2^(1/5)*3^(4/5)*x^(4/5) + 624/7*2^(1/4)*(3)^(1/2)*(65/7)^(3/4)*x^5};
		\legend{$C_\xi(\chi \lVert v_0 \rVert_{L^\infty(\Omega)})$};
		\addlegendentry{$C_\mu(\chi \lVert v_0 \rVert_{L^\infty(\Omega)})$};
		\draw[fill] (axis cs:0.38,3.9) node[below right] {$\chi \lVert v_0 \rVert_{L^\infty(\Omega)}$};
		\draw[fill] (axis cs:{0.353,7.60}) circle [radius=1.7pt] node[above left] {};
		\draw [dashed] (axis cs:0.353,7.60) -- (axis cs:0.353,0);
		\draw[fill] (axis cs:0.325,-0.3) node[below right] {$\rho_0$};
		\end{axis}
		\end{tikzpicture}
	}
	\subfigure[$n=6.$]  
	{  
		\begin{tikzpicture}[scale=0.8]
		\begin{axis}[
		axis lines=middle,
		clip=false,
		ymin=0,
		 xtick={0,0.08,...,1},
		legend pos=north west
		]
		\addplot+[mark=none,samples=200,domain=0:0.4] {9*(21/2)^(1/3)* x^(2/3)};
		\addplot+[mark=none,samples=200,domain=0:0.4] {3*(21/2)^(1/3)*x^(2/3) + 10752*x^6};
		\legend{$C_\xi(\chi \lVert v_0 \rVert_{L^\infty(\Omega)})$};
		\addlegendentry{$C_\mu(\chi \lVert v_0 \rVert_{L^\infty(\Omega)})$};
		\draw[fill] (axis cs:0.288,5.8) node[below right] {$\chi \lVert v_0 \rVert_{L^\infty(\Omega)}$};
		\draw[fill] (axis cs:{0.285,8.40}) circle [radius=1.7pt] node[above left] {};
		\draw [dashed] (axis cs:0.285,7.9510) -- (axis cs:0.285,0);
		\draw[fill] (axis cs:0.266,-0.8) node[below right] {$\rho_0$};
		\end{axis}
		\end{tikzpicture}
	}  
\caption{Illustration comparing for $n\in\{3,4,5,6\}$ the graphs of the functions $C_\mu(\chi \lVert v_0 \rVert_{L^\infty(\Omega)})$ and 
$C_\xi(\chi \lVert v_0 \rVert_{L^\infty(\Omega)})$  computed in \cite[Theorem 1.1]{LankeitWangConsumptLogistic} and in Theorem \ref{MainTheorem}, both for $p=\frac{n}{2}$. The intersection point is such that the abscissa $\rho_0$ approaches 0 with $n$ increasing. For values of $\chi \lVert v_0 \rVert_{L^\infty(\Omega)}$ smaller than $\rho_0$, $C_\xi$ and $C_\mu$ are rather similar (quantitatively of the same order), whilst for larger values,  $C_\mu$ is much bigger than $C_\xi$. This phenomenon is even more perceptible in higher dimensions.}\label{ComparisonConditionsLogisticRepulsionFigureN-Mezzi}
\end{figure}
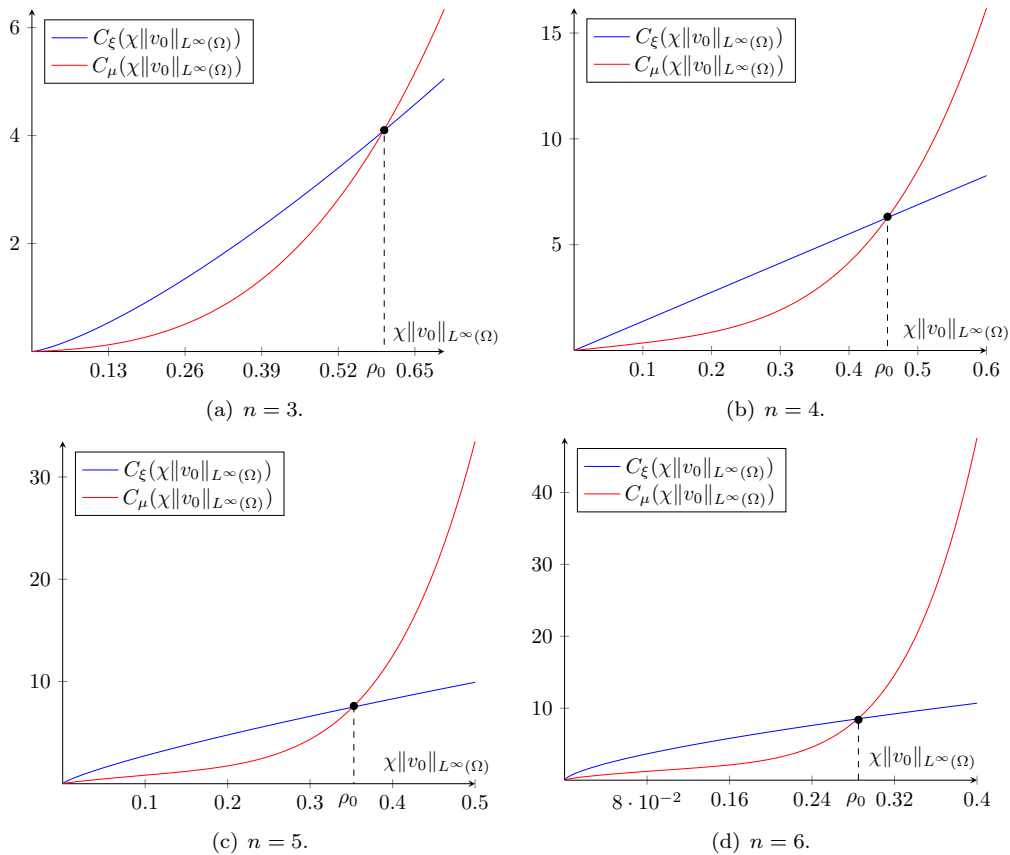

As a consequence, if we consider that for high values of the cell concentration in problem $P_\xi$ the chemoattractant is consumed with a weaker law than that in $P_\mu$ ($0<\alpha<\frac{1}{2}+\frac{1}{n}$ vs. $\alpha=1$, respectively), this discussion seems to indicate that the introduction in the classical Keller--Segel model with consumption \eqref{problemOriginalKSCosnumption} of a produced chemorepellent, has a more effective stabilizing impact on the cells' motility than the one resulting by the introduction of dampening logistic sources.
\subsubsection*{Acknowledgments}
The authors are members of the Gruppo Nazionale per l'Analisi Matematica, la Probabilit\`a e le loro Applicazioni (GNAMPA) of the Istituto Na\-zio\-na\-le di Alta Matematica (INdAM). GV is partially supported by the research projects \textit{Evolutive and stationary Partial Differential Equations with a focus on biomathematics}, funded by Fondazione di Sardegna (2019), and by MIUR (Italian Ministry of Education, University and Research) Prin 2017 \textit{Nonlinear Differential Problems via Variational, Topological and Set-valued Methods} (Grant Number: 2017AYM8XW). 

\end{document}